\documentclass[12pt]{amsart}
\usepackage{amsmath}
\usepackage{amssymb}
\usepackage{amsrefs}
\usepackage{a4wide}
\usepackage{amscd} %AMS-LaTex package for commutative diagrams
\usepackage{graphics}
\usepackage{color}
\usepackage[all]{xy}
\usepackage{amsthm} % Ams theorem style package.

\newtheorem{theorem}{Theorem}[section]
\newtheorem*{theorem*}{Theorem}
\newtheorem{lemma}[theorem]{Lemma}
\newtheorem{corollary}[theorem]{Corollary}
\newtheorem{proposition}[theorem]{Proposition}
\theoremstyle{definition}
\newtheorem{definition}[theorem]{Definition}
\newtheorem{example}[theorem]{Example}
\theoremstyle{remark}
\newtheorem{remark}[theorem]{Remark}

\numberwithin{equation}{section}
\newcommand{\AProperClass}{\ensuremath{\mathcal{P}}} %--- A
\newcommand{\To}{\longrightarrow}   %---Fonksiyon yazrken ki f:A-->B oku uzun
\newcommand{\Tor}{\operatorname{Tor}}
\newcommand{\Hom}{\operatorname{Hom}}
\newcommand{\Ext}{\operatorname{Ext}} %---- Ext yazmak icin

\newcommand{\Image}{\operatorname{Im}} %---- Im yazmak icin (Image anlaminda)
\newcommand{\Ker}{\operatorname{Ker}}
\newcommand{\Soc}{\operatorname{Soc}}
\newcommand{\coclosed}[2]{\ensuremath{\xymatrix@1{ #1 \, \ar@{^{(}->}[r]^-{cc} &
#2}}}
\newcommand{\cosmall}[3]{\ensuremath{\xymatrix@1{ #1 \,
\ar@{^{(}->}[r]^-{cs}_-{#2} & #3}}}

\newcommand{\ShortExactSequence}[5]{\ensuremath{\xymatrix@1{ 0 \ar[r] &  #1 \ar[r]^-{#2} & #3 \ar[r]^-{#4} &  #5  \ar[r] & 0 }}}
\newcommand{\LongExactSequence}[5]{\ensuremath{\xymatrix@1{ 0 \ar[r] & #5 \ar[r] &#1 \ar[r]^-{#2} & #3 \ar[r]^-{#4} &  #5  \ar[r] & 0 }}}
\newcommand{\AShortExactSequence}[6]{\ensuremath{\xymatrix@1{ #1: 0 \ar[r] &  #2 \ar[r]^-{#3} & #4 \ar[r]^-{#5} &  #6  \ar[r] & 0 }}}
\begin{document}

\author{Engin B\"{u}y\"{u}ka\c{s}{\i}k}\address{Izmir Institute of Technology, Department of Mathematics, 35430, Urla, Izmir, Turkey}\email{enginbuyukasik@iyte.edu.tr}

\author{Y{\i}lmaz Dur\u{g}un}
\address{Izmir Institute of Technology, Department of Mathematics,\\ G\"{u}lbah\c{c}ek\"{o}y\"{u}, 35430, Urla, Izmir, Turkey.}
\email{yilmazdurgun@iyte.edu.tr}

\title{Neat-flat modules}

\begin{abstract}Let $R$ be a ring and $M$ be a right $R$-module.
$M$ is called neat-flat if any short exact sequence of the form
$0\to K\to N\to M\to 0$ is neat-exact i.e. any homomorphism from a
simple right $R$-module $S$ to $M$ can be lifted to $N$. We prove
that, a module is neat-flat if and only if it is simple-projective.
Neat-flat right $R$-modules are projective if and only if $R$ is a
right $\sum$-$CS$ ring. Every finitely generated neat-flat right
$R$-module is projective if and only if $R$ is a right $C$-ring and
every finitely generated free right $R$-module is extending. Every
cyclic neat-flat right $R$-module is projective if and only if $R$
is right $CS$ and right $C$-ring.  Some characterizations of
neat-flat modules are  obtained over the rings whose simple right
$R$-modules are finitely presented.
\end{abstract}


\maketitle

\section{Introduction}

Throughout, $R$ is an associative ring with identity and all modules
are unitary right $R$-modules.
 For an R-module $M$, $M^{+}$, $E(M)$, $\Soc (M)$  will denote the character module, injective hull, the socle of $M$, respectively. A subgroup $A$ of an abelian group $B$ is called \emph{neat} in $B$ if $pA=A\cap pB$ for each prime integer $p$. The notion of neat subgroup generalized to modules by Renault (see, \cite{renault:neat}). Namely, a submodule $N$ of $R$-module $M$ is called \emph{neat} in $M$, if for every simple $R$-module $S$, every homomorphism $f:S \to M/N$ can be lifted to a homomorphism $g: S \to M$. Equivalently, $N$ is  neat in $M$ if and only if $\Hom (S, g):\, \Hom (S, M) \to \Hom (S, M/N)$ is an epimorphism for every simple $R$-module $S$.
Neat submodules have been studied extensively by many authors (see,
\cite{Wisbauer-et.al:t-complementedandt-supplementedmodules},
\cite{Fuchs:NeatSubmodulesOverIntegralDomain},
\cite{Mermut:Ph.D.tezi}, \cite{Stenstrom:highsubmodulesandpurity},
\cite{Stenstrom:Puresubmodules}). An $R$-module $M$ is called
\emph{m-injective} if for any maximal right ideal $\emph{I}$ of $R$,
any homomorphism $f : I \to M$ can be extended to a homomorphism $g
: R \to M$ (see, \cite{septimi:minjective}, \cite{Mermut:Ph.D.tezi},
\cite{Ozdemir:Ph.D.tezi},
\cite{smith:Injectivemodulesandprimeideals},
\cite{Wang:onmaximalinjective}, \cite{Xiang:maxflatmaxinjective}).
Note that, $m$-injective modules are called max-injective in
\cite{Wang:onmaximalinjective}. It turns out that, a module  $M$ is
$m$-injective if and only if $\Ext^{1}_{R}(R/\emph{I},M) = 0$ for
any maximal right ideal $\emph{I}$ of $R$ if and only if $M$ is a
neat submodule   in every module containing it i.e. any short exact
sequence of the form $0 \rightarrow M \rightarrow N\rightarrow L
\rightarrow 0 $ is neat-exact (see, \cite[Theorem
2]{septimi:minjective}). A ring $R$ is a right $C$-ring if for every
proper essential right ideal $I$ of $R$, the module $R/I$ has a
simple module, (see, \cite{Renault:Cring}). Any right semiartinian
ring is a $C$-ring, and a domain is a $C$-ring if and only if every
torsion $R$-module contains a simple module. By \cite[Lemma
4]{smith:Injectivemodulesandprimeideals}, $R$ is a right $C$-ring if
and only if every $m$-injective module is injective.

Motivated by the relation between $m$-injective modules and neat
submodules, we  investigate the modules $M$, for which any short
exact sequence ending with $M$ is neat-exact. Namely, we call $M$
\emph{neat-flat} if for any epimorphism $f:N \rightarrow M$, the
induced map $\Hom (S,\,N)\rightarrow \Hom (S, M)$ is surjective for
any simple right $R$-module $S$.

In \cite{mao:whendoeseverysimplehaveaprojectiveenvelope}, a right
$R$-module $M$ is called \emph{simple-projective} if for any simple
right $R$-module $N$, every homomorphism $f: N \rightarrow M$
factors through a finitely generated free right $R$-module $F$, that
is, there exist homomorphisms $g: N \rightarrow F$ and $h: F
\rightarrow M$ such that $f=hg.$ Simple-projective modules and a
generalization of these modules have been studied in
\cite{mao:whendoeseverysimplehaveaprojectiveenvelope} and
\cite{rada}, respectively. By using simple-projective modules, the
authors, characterize the rings whose simple (resp. finitely
generated) right modules have projective (pre)envelope in the sense
of \cite{xu:flatcoverofmodules}. Clearly,  projective modules and
modules with $\Soc (M)=0$  are  simple-projective. Also, a simple
right $R$-module is simple-projective if and only if it is
projective. Hence, $R$ is a semisimple Artinian ring if and only if
every right $R$-module is simple-projective (see, \cite[Remark
2.2.]{mao:whendoeseverysimplehaveaprojectiveenvelope}).

The paper is organized as follows.

In section 3, we prove that, a right $R$-module $M$ is neat-flat if
and only if $M$ is simple-projective (Theorem
\ref{thm.neatflat=simpleprojective}). The right socle of $R$ is zero
if and only if  neat-flat modules coincide with the modules that
have zero socle (Proposition \ref{Proposition: neat-flat=socle
zero}). We also investigate the rings over which neat-flat modules
are projective. Namely, we prove that, (1) every neat-flat module is
projective if and only if $R$ is a right $\sum$-$CS$ ring (Theorem
\ref{Hring}); (2) every finitely generated neat-flat module is
projective if and only if $R$ is a right $C$-ring and every finitely
generated free right $R$-module is extending (Theorem
\ref{finitelygeneratedneatflat}); (3) every cyclic right $R$-module
is projective if and only if $R$ is right $CS$ and right $C$-ring
(Corollary \ref{cor:cyclic neat-flat}).

In section 4, we consider  neat-flat modules over the rings whose
simple right  modules are finitely presented. In this case, the
Auslander-Bridger tranpose of any simple right $R$-module is a
finitely presented left $R$-module. This fact is used to obtain
several characterization of neat-flat modules. Also, we examine the
relation between the flat, absolutely pure and neat-flat modules
over such rings.

For the unexplained concepts and results we refer the reader to
\cite{AF}, \cite{extendingmodules},
\cite{Wisbauer:Foundationsofmoduleandringtheory} and
\cite{xu:flatcoverofmodules}.

\section{preliminaries}

The class of neat-exact sequences form a proper class in the sense
of \cite{Buchsbaum}. This fact leads to an important
characterization of neat-flat modules (see, Lemma
\ref{lemma:neat-flat}). This characterization become crucial in the
proof of the results in the present paper. In this section, we give
some definitions and results which are used in the sequel.

Let $R$ be an associative ring with identity and $\AProperClass$ be
a class of short exact sequences of right $R$-modules and $R$-module
homomorphisms. If a short exact sequence $\mathbb{E}:0\to
A\overset{f}{\to}B\overset{g}{\to} C\to 0$  belongs to
$\AProperClass$, then $f$ is said to be a
$\AProperClass$-\emph{monomorphism}
 and $g$ is said to be a $\AProperClass$-\emph{epimorphism}.
A short exact sequence $\mathbb{E}$ is determined by each of the
monomorphisms $f$ and the epimorphisms $g$ uniquely up to
isomorphism.

\begin{definition}\label{definition of proper classes} The class $\AProperClass$ is said to be \emph{proper} (in the sense of Buchsbaum) if it satisfies the
following conditions \cite{maclane:homology}:
\begin{enumerate}
    \item[P-1)] If a short exact sequence $E$ is in $\AProperClass$, then $\AProperClass$
    contains every short exact sequence isomorphic to $E$ .
    \item[P-2)] $\AProperClass$ contains all splitting short exact sequences.
    \item[P-3)] The composite of two $\AProperClass$-monomorphisms is a $\AProperClass$-monomorphism if this composite is defined.
    \item[P-4)] The composite of two $\AProperClass$-epimorphisms is a $\AProperClass$-epimorphism if this composite is defined.
    \item[P-5)] If $g$ and $f$ are monomorphisms, and $g\,\circ f$ is a $\AProperClass$-monomorphism, then $f$ is a $\AProperClass$-monomorphism.
    \item[P-6)] If $g$ and $f$ are epimorphisms, and $g\,\circ f$ is a $\AProperClass$-epimorphism.
    then $g$ is a $\AProperClass$-epimorphism.
\end{enumerate}
\end{definition}
\begin{sloppypar}From now on,  $\AProperClass$ will denote a proper class.
A module $M$ is called {\it $\AProperClass$-flat}
 if every short exact
sequence of the form $0\to A\to B\to M\to 0$ is in $\AProperClass.$
\end{sloppypar}

For a class $\mathcal{M}$ of right $R$-modules, let  $\tau
^{-1}(\mathcal{M})=\{\mathbb{E}\mid M \otimes \mathbb{E}\, \text{
exact for each} M \in  \mathcal{M} \}$, and $\pi
^{-1}(\mathcal{M})=\{\mathbb{E} \mid \Hom (M,\,\mathbb{E})\,\,
\text{is exact for each} M \in \mathcal{M} \}.$ Then the $\tau
^{-1}(\mathcal{M})$ and $\pi ^{-1}(\mathcal{M})$ are proper classes
(see, \cite{Sklyarenko:RelativeHomologicalAlgebra}). The classes
$\tau ^{-1}(\mathcal{M})$ and $\pi ^{-1}(\mathcal{M})$ are called
flatly generated and projectively generated by $\mathcal{M}$,
respectively.

\begin{theorem}\cite[Theorem 8.1]{Sklyarenko:RelativeHomologicalAlgebra}\label{conjugateofproperclasses} Let $\mathcal{M}$ be a class of modules and $\mathbb{E}$ be a short exact sequence. Then   $\mathbb{E} \in \tau ^{-1}(\mathcal{M})$ if and only if $\mathbb{E}^{+} \in \pi ^{-1}(\mathcal{M}).$
\end{theorem}

Let $M$ be a finitely presented right $R$-module. Then there is an
exact sequence $\gamma:P_{0}\overset{f}{\to}P_{1}\overset{g}{\to} M$
where $P_{0}$ and $P_{1}$ are finitely generated projective right
$R$-modules. By applying the functor $(-)^{*} = \Hom_{R}(-,R)$ to
this sequence, we get: $0\to \Hom_{R}(M,R)\overset{g^{*}}{\to}
Hom_{R}(P_{0},R)\overset{f^{*}}{\to}\Hom_{R}(P_{1},R).$ If the right
side of this sequence of left $R$-modules filled by the module
$Tr_{\gamma}(M):= Coker(f^{*}) = P_{1}^{*} / \Image(f^{*})$ then we
obtain the exact sequence $ \gamma^{*}:
P_{0}^{*}\overset{f^{*}}{\to}P_{1}^{*}\overset{\sigma}{\to}Tr_{\gamma}(M)\to
0$ where $\sigma$ is the canonical epimorphism. For a finitely
generated projective $R$-module $P$, its dual $P^{*} =
\Hom_{R}(P,R)$ is a finitely generated projective right $R$-module.
So $P_{0}^{*}$ and $P_{1}^{*}$ are finitely generated projective
modules, hence the exact sequence $ \gamma^{*}$ is a presentation
for the finitely presented right $R$-module $Tr_{\gamma}(M)$ which
is called the Auslander-Bridger tranpose of the finitely presented
$R$-module $M$, (see \cite{auslenderbridge:stablemoduletheory}).

\begin{proposition}\cite[Corollary 5.1]{Sklyarenko:RelativeHomologicalAlgebra}\label{transpoz1} For any finitely presented right $R$-module $M$ and any short exact sequence $\mathbb{E}$ of right $R$-modules, the sequence $\Hom(M,\mathbb{E})$ is exact if and only if
the sequence $\mathbb{E} \otimes Tr(M) $ is exact.
\end{proposition}

\begin{theorem}\cite[Theorem 8.3]{Sklyarenko:RelativeHomologicalAlgebra}\label{transpoz2}
 Let $\mathcal{M}$ be a set of finitely presented left $R$-modules.
Let $Tr(\mathcal{M}) = \left\{Tr(M)| M\in \mathcal{M}\right\}$. Then
we have $\pi^{-1}(\mathcal{M}) = \tau^{-1}(Tr(\mathcal{M}))$ and
$\tau^{-1}(\mathcal{M}) = \pi^{-1}(Tr(\mathcal{M})).$ \end{theorem}

\section{Neat-flat modules}

By definition, the class of neat-exact sequences is projectively
generated by the class of simple right $R$-modules. Hence neat-exact
sequences form a proper class.  For the following lemma we refer to
\cite[Proposition 1.12-1.13]{mishina:abeliangroupsandmodules}. Its
proof is included for completeness.

\begin{lemma}\label{lemma:neat-flat}The following are equivalent for a right $R$-module $M$.
\begin{enumerate}

\item[(1)] $M$ is neat-flat.
 \item [(2)]Every exact sequence $0\to A\to B\to M\to 0$ is neat exact.
 \item [(3)]There exists a neat exact sequence $0\to K\to F\to M\to 0$ with $F$ projective.
 \item [(4)]There exists a neat exact sequence $0\to K\to F\to M\to 0$ with $F$ neat-flat.
\end{enumerate}
\end{lemma}
\begin{proof}
$(1)\Rightarrow (2) \Rightarrow (3) \Rightarrow (4)$ are clear.

$(4) \Rightarrow (1)$ Let $0\to A\to B\overset{g}{\to} M\to 0$ be
any short exact sequence. We claim that $g$ is a neat epimorphism.
By (4), there exists a neat exact sequence $0\to K\to
F\overset{s}{\to} M\to 0$ with $F$ neat-flat. We obtain a
commutative diagram with exact rows
$$
\begin{xy}
  \xymatrix{
 & & & 0 \ar[r] & A\ar@{=}[d] \ar[r]& B' \ar[d]^{u}\ar[r]^{t} & F\ar[r]\ar[d]^{s} & 0 \\
& & & 0 \ar[r] & A \ar[r] & B\ar[r]^{g} & M\ar[r] & 0 & &
              }
\end{xy}
$$
 in which the right
square is a pullback diagram. Since $F$ is neat-flat, $t$ is a neat
epimorphism. Then  $gu=st$ is a neat epimorphism by \ref{definition
of proper classes} P-4), and so $f$ is a neat epimorphism by
\ref{definition of proper classes} P-6). This completes the proof.
\end{proof}

\begin{theorem}\label{thm.neatflat=simpleprojective}  Let $R$ be a ring and $M$ be an $R$-module. Then $M$ is simple-projective if and only if $M$ is neat-flat.
\end{theorem}

\begin{proof} Suppose $M$ is simple-projective and $s :R^{(I)} \rightarrow M$ be an epimorphism. Let $S$ be simple right $R$-module and $f:S \rightarrow M$ be a homomorphism. As $M$ is simple-projective $f$ factors through a finitely generated free module i.e. there are homomorphisms $h: S \rightarrow R^n$ and $g: R^n \rightarrow M$ such that $f=gh.$ Since $R^n$ is projective, there is a homomorphism $t:R^n \rightarrow R^{(I)}$ such that $g=st.$ We get the following diagram $$
\begin{xy}
  \xymatrix{
  & &R^n\ar[rd]^{g}\ar[d]^{t} &S\ar[d]^{f}\ar[l]^{h}&& & & \\
 & & R^{(I)}\ar[r]^{s}&M & & && \\
              }
\end{xy}
$$Then $f=gh=sth$, and so the induced map $\Hom(S,\,R^{(I)}) \rightarrow \Hom(S,\, M) \rightarrow 0$ is surjective.
 Therefore the sequence $0\to \Ker s\to R^{(I)}\overset{s}{\to} M\to 0$  is neat exact.
 Hence $M$ is neat-flat by  Lemma \ref{lemma:neat-flat}(3).

Conversely, let $M$ be a neat-flat module. Then there is a neat
exact sequence $0\to K\to F\overset{g}{\to} M\to 0$  with $F$ free
by Lemma \ref{lemma:neat-flat}. Let $S$ be a simple module and $f: S
\rightarrow M$ be any homomorphism. Then there is a homomorphism $h:
S \rightarrow F$ such that $f=gh.$ As $S$ is finitely generated,
$h(S) \subseteq H$ for some finitely generated free submodule of
$F$. Then we get $f=gh=(gi)h'$ where $i:H \rightarrow F$ is the
inclusion and $h':S \rightarrow H$ is the homomorphism defined as
$h'(x)=h(x)$ for each $x \in S.$ Therefore $f$ factors through $H$,
and so $M$ is simple projective.
\end{proof}

Let $M$ be a right $R$-module with $\Soc(M)=0.$ Then $\Hom(S,\,M)=0$
for any simple right $R$-module $S$, and so $M$ is neat-flat.

\begin{proposition}\label{Proposition: neat-flat=socle zero}Let $R$ be a ring and $M$ be any $R$-module. The following are equivalent:
\begin{enumerate}
\item [(1)]$\Soc(R_{R})=0$.
\item [(2)]$M$ is neat-flat right $R$-module if and only if $\Soc(M)=0$.
   \end{enumerate}
\end{proposition}
\begin{proof}$(1)\Rightarrow (2)$ Suppose $M$ is a neat-flat right  $R$-module.
Then there is a neat exact sequence $0\to K\to P\to\ M\to 0$ with
$P$ projective by
 Lemma \ref{lemma:neat-flat}. Then the sequence  $\Hom_{R}(S,P)\to\ \Hom_{R}(S,M)\to 0$ is exact
 for any simple right $R$-module $S$. We have $\Soc(P)=0$ by (1). Then $ \Hom_{R}(S,P)=0$, and so $\Soc(M)=0$.
The converse is clear.

$(2)\Rightarrow (1)$ Since every projective module is neat-flat,
$\Soc(R_{R})=0$ by (2).
\end{proof}

\begin{proposition}\label{propertiesofneatflats}\cite[Proposition 2.4]{mao:whendoeseverysimplehaveaprojectiveenvelope}The class of simple-projective right $R$-modules is closed under
extensions, direct sums, pure submodules, and direct summands.
\end{proposition}

Recall that, a submodule $N$ of a module $M$ is called \emph{closed
(or a complement) in $M$}, if $N$ has no proper essential extension
in $M$, i.e. $N \unlhd K \leq M$ implies $N=K.$ A module $M$ is said
to be an extending module or a $CS$-module if every closed
submodule of $M$ is a direct summand of $M$. $R$ is  a right $CS$
ring if $R_{R}$ is $CS$.  $M$ is called  (countably) $\sum$-$CS$
module if every direct sum of (countably many) copies of $M$ is
$CS$, (see, for example, \cite{extendingmodules}). The $\sum$-$CS$
rings were first introduced and termed as co-$H$-rings in
\cite{oshiro}. Closed submodules are neat by \cite[Proposition
5]{Stenstrom:highsubmodulesandpurity}. By \cite[Theorem
5]{generalov}, every neat submodule is closed if and only if $R$ is
a right $C$-ring.
\begin{theorem}\label{Hring}Let $R$ be a ring. The following are equivalent.
\begin{enumerate}
\item [(1)]Every neat-flat right $R$-module is projective.
\item [(2)]$R$ is a right $\sum$-$CS$ ring.
\end{enumerate}
\end{theorem}
\begin{proof}$(1)\Rightarrow (2)$ Let $P$ be a projective $R$-module and $N$ be a closed submodule of $P$. Then $N$ is a neat submodule of $P$. So that $P/N$ is neat-flat by Lemma \ref{lemma:neat-flat} and so $P/N$ is projective by (1). Therefore the sequence $0\rightarrow N \rightarrow P \rightarrow P/N\rightarrow 0$ splits, and so $N$ is a direct summand of $P$. Hence $R$ is a  $\sum$-$CS$ ring.

$(2)\Rightarrow (1)$ Every right  $\sum$-$CS$ ring is both  right
and left perfect by \cite[Theorem 3.18]{oshiro}. Hence, $R$ is a
right $C$-ring by \cite[Theorem 28.4]{AF}. Let $M$ be a neat-flat
right $R$-module. Then there is a neat exact sequence
$\mathbb{E}:\,\, 0\to K\hookrightarrow P\to\ M\to 0$ with $P$
projective by
 Lemma \ref{lemma:neat-flat}. Since $R$ is right $C$-ring, $K$ is closed in $P$ by \cite[Theorem 5]{generalov}. Hence the sequence $\mathbb{E}$ splits by (2), and so $M$ is projective.
\end{proof}

\begin{theorem}\label{finitelygeneratedneatflat}Let $R$ be a ring. The following are equivalent.
\begin{enumerate}
\item [(1)]Every finitely generated neat-flat right $R$-module is projective.
\item [(2)]$R$ is a right $C$-ring and every finitely generated free right $R$-module is extending.
\end{enumerate}
\end{theorem}
\begin{proof}$(1)\Rightarrow (2)$ Let $I$ be an essential right ideal of $R$ with $\Soc(R/I)=0$. Then $\Hom(S, R/I)=0$ for each simple right $R$-module $S$ and hence $I$ is neat ideal of $R$. So $R/I$ is neat-flat by Lemma \ref{lemma:neat-flat}. But it is projective by (1), and so $I$ is direct summand of $R$. This is contradict with essentiality of $I$ in $R$. So that $R$ is a right $C$-ring.\\
Let $F$ be a finitely generated free right $R$-module and $K$ a
closed submodule of $F$. Since every closed submodule is neat, $F/K$
is neat-flat by Lemma \ref{lemma:neat-flat}. Then $F/K$ is
projective by (1), and so $K$ is a direct summand of $F$.

$(2)\Rightarrow (1)$ Let $M$ be a finitely generated neat-flat right
$R$-module. Then there is an exact sequence $0\to \Ker
(f)\hookrightarrow F\to M\to 0$ with $F$ finitely generated free
right $R$-module. By Lemma \ref{lemma:neat-flat} $\Ker(f)$ is neat
submodule of $F$. Since $R$ is $C$-ring, $\Ker (f)$ is closed
submodule of $F$ by \cite[Theorem 5]{generalov}. Then $0\to \Ker
(f)\hookrightarrow F\to M\to 0$ is a split exact sequence. So $M$ is
projective.
\end{proof}

Following the proof of Theorem \ref{finitelygeneratedneatflat}, we
obtain the following corollary.

\begin{corollary}\label{cor:cyclic neat-flat} Every cyclic neat-flat right $R$-module is projective if and only if $R$ is both  right $CS$  and right $C$-ring.
\end{corollary}

\begin{remark}\label{socleseries}Let $M$ be a right $R$-module. Then the socle series $\{S_{\alpha}\}$ of $M$ is defined as: $S_1=\Soc(M)$, $S_{\alpha}/S_{\alpha -1}=\Soc(M/S_{\alpha -1}),$ and for a limit ordinal $\alpha$, $S_{\alpha}=\cup _{\beta < \alpha}S_{\beta}.$ Put $S=\cup \{S_{\alpha}\}$. Then, by construction $M/S$ has zero socle. $M$ is semiartinian (i.e. every proper factor of $M$ has a simple module) if and only if $S=M$ (see, for example, \cite{extendingmodules}).
\end{remark}

From the proof of Theorem \ref{Hring}, we see that the condition
that, every free right $R$-module is extending implies $R$ is a
right $C$-ring. In the following example we show that,  if every
finitely generated free right $R$-module is extending, then $R$ need
not be a right $C$-ring. Hence the right $C$-ring condition in
\ref{finitelygeneratedneatflat} is necessary.

\begin{example} Let $R$ be the ring of all linear transformations (written on the left) of an infinite dimensional vector space over a division ring. Then $R$ is prime, regular, right self-injective and $\Soc(R_R)\neq 0$ by \cite[Theorem 9.12]{Goodearl:vonneumannregularrings}. As $R$ is a prime ring, $\Soc(R_R)$ is an essential ideal of $R_R.$ Let $S$ be as in Remark \ref{socleseries}, for $M=R$. Then $S \neq R$, by \cite[Lemma 1(2)]{Clark-Smith:OnSemiartinianandinjectivityconditions}. Since $R/S$ has zero socle, $S$ is a neat submodule of $R_R$. On the other hand, $S$ is not a closed submodule of $R$, otherwise $S$ would be a direct summand of $R$ because $R$ is right self injective (i.e. extending). Therefore $R$ is not a right $C$-ring. Also, as $R$ is right self injective  $R^n$ is injective, and so extending for every $n \geq 1.$
\end{example}

\section{Rings whose simple Right modules are finitely presented}

In this section, we consider neat-flat modules over the rings whose
simple right modules are finitely presented. The reason for
considering these rings is that, the Auslander-Bridger tranpose of
simple right $R$-modules is a finitely presented left $R$-module
over such rings.

\begin{definition}
Let $R$ be a ring and $n$ a nonnegative integer. A right $R$-module
$M$ is called $n$-presented if it has a finite $n$-presentation,
i.e., there is an exact sequence $F_{n}\to F_{n-1}\to\ldots F_{1}\to
F_{0}\to M\to 0$ in which every $F_{i}$, is a finitely generated
free right $R$-module \cite{ding:onncoherentrings}.
\end{definition}

\begin{lemma}\cite[Lemma 2.7]{ding:onncoherentrings}\label{ding2.7}
Let $R$ and $S$ be rings, and $n$ a fixed positive integer. Consider
the situation $(_{R}A, _{R}B_{S}, C_{S})$ with $_RA$ $n$-presented
and $C_S$ injective. Then there is an isomorphism

 $$\Tor^{R}_{n-1}(\Hom_{_{S}}(B, C), A)\cong\Hom_{S}(\Ext^{n-1}_{R}(A, B), C)$$.

\end{lemma}

\begin{proposition}\cite[Proof of Proposition 5.3.9.]{Relativehomologicalalgebra}\label{prop. M is pure in M++}
 Every R-module $M$ is a pure submodule of a pure injective R-module $M^{++}$.
\end{proposition}

Let $M$ be a right $R$-module. $M$ is called absolutely pure (or
FP-injective)
 if $\Ext^{1}(N, M) = 0$ for any finitely presented right $R$-module $N$, i.e. $M$ is a pure submodule of its injective hull $E(M)$. For any right $R$-module $M$, the character module $M^{+}$ is a pure injective right $R$-module, (see, \cite[Proposition 5.3.7]{Relativehomologicalalgebra}).

\begin{remark}Note that, if every simple right $R$-module is finitely presented, then every pure submodule is neat. So that, in this case, any right  flat $R$-module is neat-flat.
\end{remark}

Using, Theorem \ref{transpoz2}, we obtain the following
characterization of neat-flat modules.

\begin{theorem}\label{charecteronscoherent}\begin{sloppypar}
Let $R$ be a ring such that every simple right $R$-module is
finitely presented. Then $M$ is a neat-flat right $R$-module if and
only if $\Tor_{1}(M,\, Tr(S))=0$ for each simple right $R$-module
$S$.
\end{sloppypar}
\end{theorem}
\begin{proof}\begin{sloppypar}
 Let $M$ be  an $R$-module and $\mathbb{E}:0\to K\overset{f}{\to} F\to M\to 0$
 be a short exact sequence with $F$ projective.
Let $S$ be simple right $R$-module. Tensoring $\mathbb{E}$ by $
Tr(S)$ we get the exact sequence $$0=\Tor_{1}(F,
Tr(S))\to\Tor_{1}(M, Tr(S))\to K \otimes Tr(S)\overset{f \otimes
1_{Tr(S)}}{\To}F \otimes Tr(S).$$

Now, suppose $M$ is  neat-flat. Then $\mathbb{E}$ is neat-exact by
Lemma \ref{lemma:neat-flat}. So that $f \otimes 1_{Tr(S)}$ is monic,
by Theorem \ref{transpoz2}. Hence $\Tor_{1}(M, Tr(S))=0$.

Conversely, suppose $\Tor_{1}(M,Tr(S))=0$ for each simple right
$R$-module $S$. Then the sequence $0\to K \otimes Tr(S) \to F\otimes
Tr(S)$ is exact, and so the sequence $0\to K\to F\to M\to 0$ is
neat-exact by Theorem \ref{transpoz2}. Then $M$ is neat-flat by
Lemma \ref{lemma:neat-flat}.
\end{sloppypar}
\end{proof}

\begin{corollary}\label{characterofabsolutelypure}Let $R$ be a ring such that every simple right $R$-module is finitely presented and $M$ be an arbitrary $R$-module. If $M$ is  absolutely pure, then $M^{+}$ is neat-flat.
\end{corollary}

\begin{proof}Let $S$ be a simple right $R$-module. By our assumption $S$ is finitely presented, and so $Tr(S)$ is finitely presented $R$-module.
Then $\Ext^{1}(Tr(S), M)=0$, because $M$ is absolutely pure. We
have, $0=\Ext^{1}(Tr(S), M)^{+}\cong \Tor_{1}(M^{+},Tr(S))$ by Lemma
\ref{ding2.7}. Hence $\Tor_{1}(M^{+},Tr(S))= 0$, and so $M^{+}$ is
neat-flat by Theorem \ref{charecteronscoherent}.
\end{proof}

\begin{corollary}\label{characterofinjective}Let $R$ be a ring such that every simple
 $R$-module is finitely presented and $M$ be a right $R$-module. If $M$ is injective, then $M^{+}$ is neat-flat.
\end{corollary}

\begin{lemma}\label{characterofneatflat}Let $R$ be a ring such that every simple $R$-module
 is finitely presented and $M$ be a right $R$-module. Then $M$ is neat-flat if and only if $M^{++}$ is neat-flat.
\end{lemma}

\begin{proof}
\begin{sloppypar}Let $\mathcal{M}$ be the set of all representatives of simple right $R$-modules.
 Suppose $M$ is a neat-flat $R$-module. Then there exists a neat-exact sequence
 $\mathbb{E}:0\to K\to F\to M\to 0$ with $F$ projective by Lemma \ref{lemma:neat-flat}.
 By Theorem \ref{transpoz2},  $ \mathbb{E} \in\tau^{-1}(Tr(\mathcal{M}))$.
 Then $\mathbb{E}^{+}\in\pi^{-1}(Tr(\mathcal{M}))$ by Theorem \ref{conjugateofproperclasses}, and so
 $\mathbb{E}^{+}\in\tau^{-1}(\mathcal{M})$ by Theorem \ref{transpoz2}.
 Again by Theorem \ref{transpoz2} and Theorem \ref{conjugateofproperclasses} we have
 $\mathbb{E}^{++}:0\to K^{++}\to F^{++}\to M^{++}\to 0\in\pi^{-1}(\mathcal{M})
= \tau^{-1}(Tr(\mathcal{M})).$
 \end{sloppypar}

Since $F$ is projective, $F^{+}$ is injective by \cite[Theorem
3.52]{Rotman:HomologicalAlgebra}. Thus $F^{++}$ is neat-flat by
Corollary \ref{characterofinjective}. Then  $M^{++}$ is neat-flat,
since $E^{++}$ is neat exact, and neat-flat modules closed under
neat quotient by Lemma \ref{lemma:neat-flat}.

 Conversely, suppose $M^{++}$ is neat-flat.  Since $M$ is a pure submodule of $M^{++}$ by Proposition \ref{prop. M is pure in M++}, $M$ is neat-flat by Theorem \ref{thm.neatflat=simpleprojective} and Proposition \ref{propertiesofneatflats}.
\end{proof}

\begin{definition}A right $R$-module $M$ is called \emph{max-flat} if
$\Tor_{R}^{1}(M,R/I) = 0$ for every maximal left ideal $I$ of $R$
(see, \cite{Xiang:maxflatmaxinjective}).
\end{definition}
Note that  a right $R$-module $M$ is max-flat if and only if $M^{+}$
is $m$-injective by the standard isomorphism
$\Ext^{1}(S,M^{+})\cong\Tor_{1}(M, S)^{+}$ for all simple left
$R$-module $S$.

Using the similar arguments of \cite[Theorem
4.5]{Xiang:maxflatmaxinjective}, we can prove the following. The
proof is omitted.
\begin{theorem}\label{characterization} Let $R$ be a ring such that every simple right $R$-module is
 finitely presented and $M$ be a right $R$-module. Then the followings are hold.
\begin{enumerate}
\item[(1)]$M$ is an $m$-injective right $R$-module if and only if $M^{+}$ is max-flat.
\item[(2)]$M$ is an $m$-injective right $R$-module if and only if $M^{++}$ is $m$-injective.
\item[(3)]$M$ is a max-flat right $R$-module if and only if $M^{++}$ is max-flat.
\end{enumerate}
\end{theorem}

\begin{proposition}\cite[Theorem 3]{septimi:minjective}\label{max-injective}The following are equivalent for a right $R$-module $M$:
\begin{enumerate}
\item [(1)]$M$ is an $m$-injective $R$-module.
\item [(2)]$\Soc(E(M)/M)=0$.
\end{enumerate}
\end{proposition}

\begin{proposition}\label{neat-flatflat} Assume that every  neat-flat right $R$-module is flat. Then the following are hold.
\begin{enumerate}
\item[(1)]Every $m$-injective right $R$-module is absolutely pure.
\item[(2)]For every right $R$-module $M$, $M$ is max-flat if and only if $M$ is flat.
\end{enumerate}
\end{proposition}
\begin{proof}\begin{sloppypar} (1) Let $M$ be an $m$-injective right $R$-module. By Proposition \ref{max-injective}, $\Soc(E(M)/M)=0$, and so $E(M)/M$ is neat-flat. Then $E(M)/M$ is flat by our hypothesis. Hence $M$ is a pure submodule of $E(M)$, and so $M$ is an absolutely pure module.

(2) Assume  $M$ is a max-flat right $R$-module. Then $M^{+}$ is
$m$-injective, and so it is absolutely pure by (1). But $M^{+}$ pure
injective by \cite[Proposition 5.3.7]{Relativehomologicalalgebra},
so $M^{+}$ is injective. Then $M$ is flat by \cite[Theorem
3.52]{Rotman:HomologicalAlgebra}. The converse statement is clear.
\end{sloppypar}
\end{proof}

\begin{theorem}\cite[Theorem 1]{Flatandprojectivecharactermodules}\label{coherentring} The following statements are equivalent:
\begin{enumerate}
\item[(1)] $R$ is a right coherent ring.
\item[(2)] $M_{R}$ is absolutely pure if and only if $M^{+}$ is a flat module.
\item[(3)] $M_{R}$ is absolutely pure if and only if $M^{+ +}$ is an injective left R-module.
\item[(4)] $_{R}M$ is flat if and only if $M ^{+ +}$ is a flat left R-module.
\end{enumerate}
\end{theorem}

\begin{proposition}\label{neat-flatflatandfpsimple} Consider the following statements.
\begin{enumerate}
\item[(1)] Every neat-flat right $R$-module is flat, and every simple right $R$-module is finitely presented.
\item[(2)]$M$ is an $m$-injective right $R$-module if and only if $M^{+}$ is a flat left $R$-module.
   \item[(3)] $R$ is a right coherent ring, and $M$ is an $m$-injective right $R$-module if and only if $M$ is an absolutely pure right $R$-module.

\end{enumerate}
Then $(1)\Rightarrow (2)\Leftrightarrow(3).$
\end{proposition}
\begin{proof}
$(1)\Rightarrow (3)$ By Proposition \ref{neat-flatflat}(1), every
$m$-injective right $R$-module is absolutely pure. On the other
hand, every absolutely pure right $R$-module is $m$-injective since
every simple right $R$-module is finitely presented by (1). Then,
for every right $R$-module $M$, $M$ is absolutely pure if and only
if $M$ is $m$-injective, if and only if $M^{+}$ is max-flat by
Theorem \ref{characterization}(2), if and only if $M^{+}$ is a flat
module by Proposition \ref{neat-flatflat}(2). Hence $R$ is a right
coherent ring by \cite[Theorem
1]{Flatandprojectivecharactermodules}. This proves (3).

$(2)\Rightarrow(3)$ Let $M$ be a left $R$-module. We claim that, $M$
is a flat  $R$-module if and only if $M^{++}$ is a flat module.
 If $M$ is flat, then $M^{+}$ is injective by \cite[Theorem 3.52]{Rotman:HomologicalAlgebra}, and so $M^{++}$ is flat left $R$-module by (2). Conversely, if $M^{++}$ is a flat module, then $M$ is flat since $M$ is a pure submodule of $M^{++}$ by Proposition \ref{prop. M is pure in M++} and flat modules are closed under pure submodules (see, \cite[Corollary 4.86]{Lam:lecturesonmodulesandrings}). So $R$ is a right coherent ring by Theorem \ref{coherentring}. The last part of (3) follows by (2) and Theorem \ref{coherentring} again.

$(3)\Rightarrow(2)$ By Theorem \ref{coherentring}.
\end{proof}

\begin{proposition}\label{characterofneatflat}Let $R$ be a ring such that every simple right $R$-module is finitely presented.
The following statements are equivalent:
\begin{enumerate}
\item[(1)]$M$ is an absolutely pure left $R$-module if and only if $ \Ext_{R}^{1}(Tr(S), M)=0$ for each simple right $R$-module $S$.
\item[(2)]$M$ is a flat right $R$-module if and only if $M$ is a neat-flat $R$-module.
\end{enumerate}
\end{proposition}
 \begin{proof}\begin{sloppypar}
$(1)\Rightarrow (2)$ Let $M$ be a neat-flat right $R$-module. Then
$\Tor_{1}(M,\,Tr(S))=0$ for each simple right $R$-module $S$ by
Theorem \ref{charecteronscoherent}. By the standard adjoint
isomorphism we have, $\Ext^{1}(Tr(S), M^{+}) \cong \Tor_{1}(
M,Tr(S))^{+}=0$. Then $M^{+}$ is absolutely pure left $R$-module by
(1). But $M^{+}$ pure injective, so $M^{+}$ is injective. Then $M$
is flat by \cite[Theorem 3.52]{Rotman:HomologicalAlgebra}. The
converse is clear.

$(2)\Rightarrow (1)$ Let $M$ be a $R$-module such that $
\Ext^{1}(Tr(S), M)=0$ for each simple $R$-module $S$.
 Then, by  Lemma \ref{ding2.7}, $0= \Ext^{1}(Tr(S), M)^{+}=\Tor_{1}(M^{+},Tr(S))$.
 So, $M^{+}$ is neat-flat by Theorem \ref{charecteronscoherent}, and it is flat by (2).
 But $R$ is right coherent by Proposition \ref{neat-flatflatandfpsimple}, so $M$ is absolutely pure by Theorem \ref{coherentring}. The converse is clear.
\end{sloppypar}
\end{proof}

%\begin{theorem}The following are equivalent.

%\begin{enumerate}
%\item[(1)] $R$ is right PS.
%\item[(2)] Submodules of neat-flat modules are neat-flat.
%\item[(3)] $\Soc(M)$ is projective for every neat-flat module M.

%\end{enumerate}
%\end{theorem}

\centerline{\bf{Acknowledgments}}

Some part of this paper was written while the second author was
visiting Padua University, Italy. He wishes to thank the members of
the Department of Mathematics for their kind hospitality and the
Scientific and Technical Research Council of Turkey (T\"{U}B\.ITAK)
for their financial support.

\end{document}